\documentclass[review,onefignum,onetabnum]{siamart190516}

\usepackage{lipsum}
\usepackage{amsfonts}
\usepackage{graphicx}
\usepackage{epstopdf}
\usepackage{algorithmic}
\ifpdf
  \DeclareGraphicsExtensions{.eps,.pdf,.png,.jpg}
\else
  \DeclareGraphicsExtensions{.eps}
\fi

\newsiamremark{remark}{Remark}
\newsiamremark{hypothesis}{Hypothesis}
\crefname{hypothesis}{Hypothesis}{Hypotheses}
\newsiamthm{claim}{Claim}

\headers{Stability equivalence}{M. H. Song, Y. D. Geng, and M. Z. Liu}

\title{Stability equivalence among stochastic differential equations and stochastic differential equations with piecewise continuous arguments and corresponding Euler-Maruyama methods\thanks{Submitted to the editors DATE.
\funding{This work is supported by the National Natural Science Foundation of China (No. 11671113)}}}

\author{Minghui Song\thanks{Department of mathematics, Harbin Institute of Technology, Harbin 150001, China
  (\email{songmh@hit.edu.cn, ydgeng@hit.edu.cn, mzliu@hit.edu.cn}).}
\and Yidan Geng\footnotemark[2]
\and Mingzhu Liu\footnotemark[2]}
\usepackage{amsopn}

\ifpdf
\hypersetup{
  pdftitle={Stability equivalence among stochastic differential equations and stochastic differential equations with piecewise continuous arguments and corresponding Euler-Maruyama methods},
  pdfauthor={M. H. Song, Y. D. Geng, and M. Z. Liu}
}
\fi
\usepackage{framed}
\usepackage[numbers,sort&compress]{natbib}
\usepackage{cleveref}

\newsiamthm{assumption}{Assumption} 
\begin{document}

\maketitle

\begin{abstract}
 In this paper, we consider the equivalence of the $p$th moment exponential stability for stochastic differential equations (SDEs), stochastic differential equations with piecewise continuous arguments (SDEPCAs) and the corresponding Euler-Maruyama methods EMSDEs and EMSDEPCAs. We show that if one of the SDEPCAs, SDEs, EMSDEs and EMSDEPCAs is $p$th moment exponentially stable, then any of them is $p$th moment exponentially stable for a sufficiently small step size $h$ and $\tau$ under the global Lipschitz assumption on the drift and diffusion coefficients
\end{abstract}



\section{Introduction}
Stochastic differential equations (SDEs) have been widely used in many branches of science and industry \cite{Arnold1974, Friedman1976, Mao2008, Lawrence2013, Braumann, Bernt2013}. There is an extensive literature in stochastic stability (e.g. the moment exponential stability or almost sure exponential stability) \cite{Arnold1974, Friedman1976, Khasminskii1980Stochastic, Mao1994stability, Schurz1997, Shaikhet, Buckwar2006}. One of the powerful techniques in the study of stochastic stability is the method of Lyapunov functions. In the absence of an appropriate Lyapunov function, we may carry out careful numerical simulations using a numerical method, say the Euler-Maruyama (EM) method \cite[see e.g.][]{Peter1992, Mao2003, Tudor1987, Milstein1995, Hutzenthaler2011,Hutzenthaler2013, Higham2002, Baker2000} with a small step size. Does the main question arise whether the numerical solutions can reproduce and predict the stability of the underlying solutions?

The case that stochastic stability of the general nonlinear equation and that of the numerical method are equivalent for a sufficiently small step size can be founded in 
\cite{Mao2003LMS, HighamNumerische2007, Mao2007SDDE, MaoSIAM2015, Deng2019, Linna2018}
, while for the linear equation in \cite{Baker2005, Higham2000, Mitsui1996}. Higham et al. in \cite{HighamSIAM2007} showed that when the SDE obeys a linear growth condition, the EM method recovers almost surely exponential stability.

In this paper, we consider the following stochastic differential equation with piecewise continuous argument (SDEPCA)
\begin{equation}\label{SDEPCA}
dx(t)=\left(f(x(t))+u_{1}\left(x\left(\left[t/\tau\right)\tau\right)\right)\right]dt+(g(x(t))+u_{2}(x([t/\tau]\tau)))dw(t)
\end{equation}
and the stochastic differential equation (SDE)
\begin{equation}\label{SDE}
dy(t)=(f(y(t))+u_{1}(y(t)))dt+(g(y(t))+u_{2}(y(t)))dw(t).
\end{equation}
We also consider the applications of EM method to SDEPCA \eqref{SDEPCA} and SDE \eqref{SDE}, respectively
\begin{equation}\label{SDEPCAs-Euler}
X_{n+1}=X_{n}+(f(X_{n})+u_{1}(X_{[n/m]m}))h+(g(X_{n})+u_{2}(X_{[n/m]m}))\Delta w_{n},
\end{equation}
\begin{equation}\label{SDEs-Euler}
Y_{n+1}=Y_{n}+(f(Y_{n})+u_{1}(Y_{n}))h+(g(Y_{n})+u_{2}(Y_{n}))\Delta w_{n},
\end{equation}
where $h=\frac{\tau}{m}, m\in\mathbb{N}^{+}$. We refer to \eqref{SDEPCAs-Euler} and \eqref{SDEs-Euler} by the terms EMSDEPCA \eqref{SDEPCAs-Euler} and EMSDE \eqref{SDEs-Euler}, respectively.  The main purpose of the present paper is to show that 
if one of the SDEPCAs \eqref{SDEPCA}, SDEs \eqref{SDE}, EMSDEPCA \eqref{SDEPCAs-Euler} and  EMSDE \eqref{SDEs-Euler} is $p$th moment exponential stable, then so are the others for a sufficiently small step size $h$ and $\tau$ under a global Lipschitz assumption on the drift and diffusion coefficients. 
 In order to do this, we shall concentrate on the following questions:
 \begin{enumerate}
\item[(Q1)] If for a sufficiently small $\tau$, the SDEPCA \eqref{SDEPCA} is $p$th moment exponentially stable, can we confidently infer that the SDE \eqref{SDE} is $p$th moment exponentially stable?
\item[(Q2)] For a sufficiently small step size $h$, does the EMSDE \eqref{SDEs-Euler} reproduce the $p$th moment exponential stability of the underlying SDE \eqref{SDE}?
\item[(Q3)] For a sufficiently small $\tau$, the EMSDEPCA \eqref{SDEPCAs-Euler} can preserve the $p$th moment exponential stability of EMSDE \eqref{SDEs-Euler}?
\item[(Q4)] If the EMSDEPCA \eqref{SDEPCAs-Euler} is $p$th moment exponentially stable, will the SDEPCA \eqref{SDEPCA} be the $p$th moment exponentially stable for a sufficiently small step size $h$?
\end{enumerate}
It is known that the positive answer to (Q2) for SDE in case $p=2$ can be founded in \cite{Mao2003LMS}. The stochastic differential equation with piecewise continuous arguments (SDEPCA) has been studied extensively \cite[see e.g.][]{MaoAutomatica2013, Maoletter2014, MaoIEEEtac2016, Dong2017, Liyuyuan2017, You2015}, and in the case of $\tau=1$, we refer to \cite{Lu2019, Lu2018}. Mao in \cite{MaoAutomatica2013} is the first paper that investigated the mean square exponentially stable for SDEPCA. The positive answer to the converse problem of (Q1), we refer to \cite{GuoSiam2016, MaoAutomatica2013, SongJMAA2018, Maoletter2014}. 

In this paper, we will give the positive answer for (Q1), (Q2), (Q3), (Q4). In \cref{sec 2}, we describe the SDEPCA and EM methods along with the definitions of $p$th moment exponential stability for SDE, SDEPCA, EMSDE, EMSDEPCA. \Cref{sec 3}, \cref{sec 4}, \cref{sec 5}, \cref{sec 6} answer the questions (Q1), (Q3), (Q4), (Q2) respectively, the final conclusions are stated in the last section.
\section{Perilimaries}\label{sec 2}
Throughout this paper, unless otherwise specified, we will use the following notations. If $\textbf{A}$ is a vector or matrix, its transpose is denoted by $\textbf{A}^{\textrm{T}}$. If $x\in\textbf{R}^{n}$, then $| x|$ is the Euclidean norm. If $\textbf{A}$ is a matrix, we let $| \mathbf{A}| =\sqrt{trace(\mathbf{A}^{{\rm{T}}}\mathbf{A})}$ be its trace norm. If $D$ is a set, its indicator function is denoted by ${\bf{1}}_{D}$. Moreover, let ($\Omega,\mathcal{F},\mathbb{P}$) be a complete probability space with a filtration $\{\mathcal{F}_{t}\}_{{t}\geq0}$ satisfying the usual conditions (that is, it is right continuous and increasing while $\mathcal{F}_{0}$ contains all $\mathbb{P}$-null sets), and let $\mathbb{E}$ denote the expectation corresponding to $\mathbb{P}$. Let $B(t)$ be a $m$-dimensional Brownian motion defined on the space. Throughout this paper, we set $p\ge 2$.

In this paper, we deal with the following $d$-dimensional nonlinear stochastic differential equations with piecewise continuous arguments (SDEPCAs)
\begin{equation}\label{SDEPCAs}
\begin{cases}
dx(t)=\left[f(x(t))+u_{1}\left(x\left(\left[t/\tau\right]\tau\right)\right)\right]dt+[g(x(t))+u_{2}(x([t/\tau]\tau))]dw(t)\\
x(0)=x_{0}\in\mathbb{R}^{d}
\end{cases}
\end{equation}
on $t\ge 0$, where $w(t)$ is an $m$-dimensional Brownian motion,  $f:\mathbb{R}^{d}\rightarrow \mathbb{R}^{d}$, $g:\mathbb{R}^{d}\rightarrow \mathbb{R}^{d\times m}$, $u_{1}:\mathbb{R}^{d}\rightarrow \mathbb{R}^{d}$ and $u_{2}:\mathbb{R}^{d}\rightarrow \mathbb{R}^{d\times m}$. $\tau$ is a positive constant, $[t/\tau]$ is the integer part of $t/\tau$. We denote $x(t)$ the solution of \eqref{SDEPCAs} with initial data $x(0)=x_{0}$ and $y(t)$ the solution of the following SDEs
\begin{equation}\label{SDEs}
dy(t)=[f(y(t))+u_{1}(y(t))]dt+[g(y(t))+u_{2}(y(t))]dw(t)
\end{equation}
on $t\ge 0$ with initial data $y(0)=x_{0}$. 

In the present paper, we also deal with the application of EM method to SDEPCA \eqref{SDEPCAs} and SDE \eqref{SDEs}. We note that $[t/\tau]\tau=n\tau$ for $t\in[n\tau, (n+1)\tau)$, $n=0, 1, 2, \cdots$, a natural choice for $h$ is $h=\frac{\tau}{m}$, $m\in\mathbb{N}^{+}$. Hence, we have
\begin{equation}\label{SDEPCAs-Euler-1}
X_{n+1}=X_{n}+(f(X_{n})+u_{1}(X_{[n/m]m}))h+(g(X_{n})+u_{2}(X_{[n/m]m}))\Delta w_{n},
\end{equation}
\begin{equation}\label{SDEs-Euler-1}
Y_{n+1}=Y_{n}+(f(Y_{n})+u_{1}(Y_{n}))h+(g(Y_{n})+u_{2}(Y_{n}))\Delta w_{n},
\end{equation}
where $X_{n}$ and $Y_{n}$ are the approximations of $x(t)$ and $y(t)$ at grid points $t=t_{n}=nh$, $n=0, 1, 2, \cdots$, respectively, $\Delta w_{n}=w(t_{n+1})-w(t_{n})$. Let $n=km+l$, $k\in\mathbb{N}^{+}$, $l=0, 1, \cdots, m-1$. Then \eqref{SDEPCAs-Euler-1} and \eqref{SDEs-Euler-1} would reduce to 
\begin{equation}\label{SDEPCAs-Euler-2}
X_{km+l+1}=X_{km+l}+(f(X_{km+l})+u_{1}(X_{km}))h+(g(X_{km+l})+u_{2}(X_{km}))\Delta w_{km+l},
\end{equation}
\begin{equation}\label{SDEs-Euler-2}
Y_{km+l+1}=Y_{km+l}+(f(Y_{km+l})+u_{1}(Y_{km+l}))h+(g(Y_{km+l})+u_{2}(Y_{km+l}))\Delta w_{km+l}.
\end{equation}
\begin{remark}\label{rem1}
If we choose $h=\tau$, then \eqref{SDEPCAs-Euler-1} and \eqref{SDEs-Euler-1} are the same and \eqref{SDEPCAs-Euler-2} and \eqref{SDEs-Euler-2} are the same.
\end{remark}

In spite of the simplicity of the EM method, explicit EM method is the most popular for approximating the solution of the SDE under global Lipschitz condition \cite[see][]{Higham2002, Peter1992, Milstein1995} and has often been used successfully in actual calculations. For further analysis it is more convenient to use continuous-time approximations,
\begin{equation}\label{eq_x_delta}
x_{\Delta}(t)=x_{0}+\int_{0}^{t} f(\bar{x}_{\Delta}(s))+u_{1}(\bar{x}_{\Delta}([s/\tau]\tau))ds+\int_{0}^{t} g(\bar{x}_{\Delta}(s))+u_{2}(\bar{x}_{\Delta}([s/\tau]\tau))dw(s),
\end{equation}
\begin{equation}\label{eq_y_delta}
y_{\Delta}(t)=x_{0}+\int_{0}^{t}f(\bar{y}_{\Delta}(s))+u_{1}(\bar{y}_{\Delta}(s))ds+\int_{0}^{t}g(\bar{y}_{\Delta}(s))+u_{2}(\bar{y}_{\Delta}(s))dw(s),
\end{equation}
where
\begin{equation*}
\bar{x}_{\Delta}(t)=\sum_{n=0}^{\infty}X_{n}{\bf{1}}_{[t_{n}, t_{n+1})}(t),\quad  \bar{y}_{\Delta}(t)=\sum_{n=0}^{\infty}Y_{n}{\bf{1}}_{[t_{n}, t_{n+1})}(t),\quad \forall t\ge 0.
\end{equation*}
We observe that $x_{\Delta}(t_{n})=\bar{x}_{\Delta}(t_{n})=X_{n}$ and $y_{\Delta}(t_{n})=\bar{y}_{\Delta}(t_{n})=Y_{n}$. Consequently, 
\begin{equation*}
x_{\Delta}([t/\tau]\tau)-\bar{x}_{\Delta}([t/\tau]\tau)=0,\quad
y_{\Delta}([t/\tau]\tau)-\bar{y}_{\Delta}([t/\tau]\tau)=0.
\end{equation*}
In this paper, we impose the following standing hypothesis.
\begin{assumption}\label{assumption 1}
Assume that there exists a positive constant $K$ such that 
\begin{equation*}
\begin{split}
| f(x)-f(y)|\vee| g(x)-g(y)|&\vee| u_{1}(x)-u_{1}(y)|\vee| u_{2}(x)-u_{2}(y)| \le K| x-y|,
\end{split}
\end{equation*}
for all $x, y\in\mathbb{R}^{d}$. Assume also that $f(0)=0$, $g(0)=0$, $u_{1}(0)=0$ and $u_{2}(0)=0$.
\end{assumption}
\Cref{assumption 1} implies that
\begin{equation*}
| f(x)|\vee| g(x)|\vee| u_{1}(x)|\vee| u_{2}(x)|\le K| x|
\end{equation*}
for all $x\in\mathbb{R}^{d}$.

We now give our basic definitions, which is cited from \cite{Mao2008}.
\begin{definition}\label{def1}
The equations SDEPCA \eqref{SDEPCAs} and SDE \eqref{SDEs} are said to be $p$th moment exponentially stable if there exist positive constants $M_{1}$, $\gamma_{1}$, $M_{2}$ and $\gamma_{2}$ such that
\begin{equation}\label{exp-sta-x}
\mathbb{E}|x(t)|^{p}\le M_{1}|x_{0}|^{p}e^{-\gamma_{1} t},\quad \forall t\ge 0,
\end{equation}
and 
\begin{equation}\label{exp-sta-y}
\mathbb{E}|y(t)|^{p}\le M_{2}|x_{0}|^{p}e^{-\gamma_{2} t},\quad \forall t\ge 0,
\end{equation}
for any $x_{0}\in\mathbb{R}^{d}$.
\end{definition}

\begin{definition}\label{def2}
For any given step size $h>0$, the Euler-Maruyama numerical methods EMSDEPCA \eqref{SDEPCAs-Euler-1} and EMSDE \eqref{SDEs-Euler-1} are said to be $p$th moment exponentially stable, if there exist positive constants $\lambda_{1}$, $L_{1}$, $\lambda_{2}$ and $L_{2}$ such that 
 \begin{equation}\label{stability}
 \mathbb{E}| X_{n}|^{p} \le L_{1}| x_{0}|^{p}e^{-\lambda_{1}nh},
 \end{equation}
 \begin{equation}\label{exp-sta-Y}
  \mathbb{E}| Y_{n}|^{p}\le L_{2}| x_{0}|^{p}e^{-\lambda_{2}nh},
 \end{equation}
 for any $x_{0}\in\mathbb{R}^{d}$, $n\in\mathbb{N}$.
 \end{definition}
 
 It is known that under \Cref{assumption 1}, for any initial value $x_{0}$ given at time $t=0$, the SDEPCA \eqref{SDEPCAs} and SDE \eqref{SDEs} have a unique continuous solutions on $t\ge 0$ (see \cite{Mao2008}). To emphasize the role of the initial value, we denote the solution $x(t)$ and $y(t)$ by $x(t;0, x_{0})$ and $y(t; 0, x_{0})$, respectively. Of course, we may consider a more general case, for example, where the SDEs and the SDEPCAs have a random initial data $x(0)=\xi$ which is an $\mathcal{F}_{0}$-measurable $\mathbb{R}^{d}$-valued random variable such that $\mathbb{E}|\xi|^{p}<\infty, \forall\ p\ge 0$. In this case, by the Markov property of the solution, we can easily see that the solution satisfies 
 \begin{equation*}
 \mathbb{E}|x(t)|^{p}=\mathbb{E}(\mathbb{E}(|x(t)|^{p}|\mathcal{F}_{0}))\le \mathbb{E}(M_{1}|\xi|^{p}e^{-\gamma_{1}t})=M_{1}\mathbb{E}|\xi|^{p}e^{-\gamma t}.
 \end{equation*}
It is therefore clear why it is enough to consider only the deterministic initial value $x(0)=x_{0}$. 

Let $y(t; s,y(s))$ be the solution of SDE \eqref{SDEs} for $t>s$ with initial value $y(s)$. It is also known that the solutions to SDE \eqref{SDEs} have the following flow property,
\begin{equation*}
y(t;0,x_{0})=y(t;s,y(s)),\quad \forall \ t\ge s>0.
\end{equation*}
Moreover, the solutions of SDE \eqref{SDEs} also have the time-homegeneous Markov property. Hence \eqref{exp-sta-y} implies
\begin{equation*}
\mathbb{E}|y(t; s,\xi)|^{p}\le M_{2}\mathbb{E}|\xi|^{p}e^{-\gamma_{2}(t-s)},\quad \forall\ t\ge s.
\end{equation*}
Given $y_{k}$ for some $k\in\mathbb{N}^{+}$, the process $\{y_{n}\}_{n\ge k}$ can be regard as the process which is produced by EM method applied to the SDE \eqref{SDEs} on $t\ge kh$ with the initial value $y(kh)=y_{k}$. In other words, the process $\{y_{n}\}_{n\ge k}$ is time-homogeneous Markov process. Hence, \cref{exp-sta-Y} is equivalent to the following more general form.
\begin{equation}\label{Y-markov}
\mathbb{E}|y_{n}|^{p}\le L_{2}\mathbb{E}|y_{k}|^{p}e^{-\lambda_{2}(n-k)h}.
\end{equation}
Due to the special feature of the SDEPCA \eqref{SDEPCAs}, the solution $x(t)$ has flow property and the Markov property at the discrete time $t=k\tau\ (k\in\mathbb{N}^{+})$. Hence 
\begin{equation*}
x(t;0,x_{0})=x(t;k\tau,x(k\tau))
\end{equation*}
and \eqref{exp-sta-x} implies
\begin{equation}\label{x-markov}
\mathbb{E}|x(t)|^{p}\le M_{1}\mathbb{E}|x(k\tau)|^{p}e^{-\gamma_{1}(t-k\tau)},\quad \ t\ge k\tau.
\end{equation}
Given $x_{km}$ for some $k\in\mathbb{N}^{+}$, the process $\{x_{n}\}_{n\ge km}$ can be regard as the process which is produced by EM method applied to the SDEPCA \eqref{SDEPCAs} on $t\ge k\tau$ with the initial value $x(k\tau)=x_{km}$. The process $\{x_{n}\}_{n\ge km}$ is time-homogeneous Markov process. Hence, \eqref{stability} is equivalent to the following more general form.
\begin{equation}\label{X_markov}
\mathbb{E}|x_{n}|^{p}\le L_{1}\mathbb{E}|x_{km}|^{p}e^{-\lambda_{1}(n-km)h}.
\end{equation}

\section{SDE \eqref{SDEs} shares the stability with SDEPCA \eqref{SDEPCAs}}\label{sec 3}

In this section, we shall investigate that if the SDEPCA \eqref{SDEPCAs} is $p$th moment exponentially stable with a sufficiently small $\tau$, then the SDE \eqref{SDEs} is also $p$th moment exponentially stable, i.e. give the positive answer to (Q1). To show this, we need several lemmas. The last lemma estimates the difference in the $p$th moment between the solution of the SDE \eqref{SDEs} and that of the SDEPCA \eqref{SDEPCAs}.
\begin{lemma}\label{lem21}
Assume that \Cref{assumption 1} holds. Then for any given constant $T\ge 0$, we have
\begin{equation}\label{eq21}
\sup_{0\le t\le T}\mathbb{E}| x(t)|^{p}\le H_{1}(T, p, K)| x_{0}|^{p},
\end{equation}
where $H_{1}(T, p, K)=e^{2pK[1+(p-1)K]T}$.
\end{lemma}
\begin{proof}
In view of It\^{o} formula and \Cref{assumption 1}, we obtain
\begin{equation*}
\begin{split}
\mathbb{E}| x(v)|^{p}\le &| x_{0}|^{p}+\mathbb{E}\int_{0}^{v}p|x(s)|^{p-1}|f(x(s))+u_{1}(x([s/\tau]\tau))|\\
&+\frac{p(p-1)}{2}|x(s)|^{p-2}|g(x(s))+u_{2}(x([s/\tau]\tau))|^{2}ds\\
\le& | x_{0}|^{p}+\mathbb{E}\int_{0}^{v}pK|x(s)|^{p-1}(|x(s)|+|x([s/\tau]\tau)|)\\
&+ p(p-1)K^{2}|x(s)|^{p-2}(|x(s)|^{2}+|x([s/\tau]\tau)|^{2}ds\\
\le& |x_{0}|^{p}+2pK[1+(p-1)K]\int_{0}^{v}\sup_{0\le u\le s}\mathbb{E}|x(u)|^{p}ds
\end{split}
\end{equation*}
Taking the supremum value of both sides over $v\in[0, t]$, we have
\begin{equation*}
\sup_{0\le v\le t}\mathbb{E}| x(v)|^{p}\le | x_{0}|^{p}+2pK[1+(p-1)K]\int_{0}^{t}\sup_{0\le u\le s}\mathbb{E}| x(u)|^{p}ds.
\end{equation*}
The desired result \eqref{eq21} follows from the well-known Gronwall inequality.
\end{proof}

\begin{lemma}\label{lem32}
Assume that \Cref{assumption 1} holds. Then for any $t\ge 0$,
\begin{equation*}\label{eq24}
\mathbb{E}| x(t)-x([t/\tau]\tau)|^{p}\le C_{1}(K, p, \tau)\tau^{\frac{p}{2}}e^{2pK[1+(p-1)K]t}| x_{0}|^{p}.
\end{equation*}
where $C_{1}(K, p, \tau)=2^{2p-1}K^{p}\left[\tau^{\frac{p}{2}}+\left(p(p-1)/2\right)^{\frac{p}{2}}\right]$.
\end{lemma}
\begin{proof}
By basic inequality, H$\ddot{o}$lder inequality, moment inequality and \Cref{assumption 1},  we obtain
\begin{equation*}
\begin{split}
\mathbb{E}| x(t)-x([t/\tau]\tau)|^{p}\le& 2^{p-1}\tau^{p-1}\mathbb{E}\int_{[t/\tau]\tau}^{t}| f(x(s))+u_{1}(x([s/\tau]\tau))|^{p}ds\\
&+ 2^{p-1}\left(\frac{p(p-1)}{2}\right)^{\frac{p}{2}}\tau^{\frac{p-2}{2}}\mathbb{E}\int_{[t/\tau]\tau}^{t}| g(x(s))+u_{2}(x([s/\tau]\tau))|^{p}ds\\
\le& C_{1}(K, p, \tau)\tau^{\frac{p}{2}-1}\int_{[t/\tau]\tau}^{t}\left(\sup_{0\le u\le s}\mathbb{E}| x(u)|^{p}\right)ds
\end{split}
\end{equation*}
It comes from \eqref{eq21} that 
\begin{equation*}\label{t_tau_mean}
\begin{split}
\mathbb{E}| x(t)-x([t/\tau]\tau)|^{p}
\le& C_{1}(K, p, \tau)\tau^{\frac{p}{2}-1}\int_{[t/\tau]\tau}^{t}e^{2pK[1+(p-1)K]s}| x_{0}|^{p}ds\\
\le& C_{1}(K, p, \tau)\tau^{\frac{p}{2}}e^{2pK[1+(p-1)K]t}| x_{0}|^{p}
\end{split}
\end{equation*}
The lemma is proved.
\end{proof}
The following lemma estimates the difference in the $p$th moment between $x(t)$ and $y(t)$.
\begin{lemma}\label{x_y_error_moment}
Let \Cref{assumption 1} hold. Then 
\begin{equation*}
\mathbb{E}| x(t)-y(t)|^{p}\le C_{2}(K, p, \tau)\tau
^{\frac{p}{2}}|x_{0}|^{p}\left(e^{C_{3}(p, K)t}-1\right),
\end{equation*}
for all $x_{0}\in \mathbb{R}^{d}$ and $t\ge 0$, where $C_{2}$ and $C_{3}$ are defined as \eqref{C_2} and \eqref{C_3}, respectively.
\end{lemma}

\begin{proof}
Using It\^{o} formula and \Cref{assumption 1}, we have
\begin{eqnarray*}
&&\mathbb{E}| x(t)-y(t)|^{p}\nonumber\\
&\le& \mathbb{E}\int_{0}^{t}pK|x(s)-y(s)|^{p-1}(|x(s)-y(s)|+|x([s/\tau]\tau)-y(s)|)\nonumber\\
&&+p(p-1)K^{2}|x(s)-y(s)|^{p-2}(|x(s)-y(s)|^{2}+|x([s/\tau]\tau)-y(s)|^{2})ds\nonumber\\
&=& \left(pK+p(p-1)K^{2}\right)\int_{0}^{t}\mathbb{E}|x(s)-y(s)|^{p}ds\nonumber\\
&&+pK\mathbb{E}\int_{0}^{t}|x(s)-y(s)|^{p-1}|x([s/\tau]\tau)-y(s)|ds\nonumber\\
&&+p(p-1)K^{2}\mathbb{E}\int_{0}^{t}|x(s)-y(s)|^{p-2}|x([s/\tau]\tau)-y(s)|^{2}ds\nonumber\\
\end{eqnarray*}
By Young inequality, we have
\begin{eqnarray}\label{x_y_2}
&&\mathbb{E}| x(t)-y(t)|^{p}\nonumber\\
&\le& \left[(2p-1+2^{p-1})K+2(p-1)(p-1+2^{p-1})K^{2}\right]\int_{0}^{t}\mathbb{E}|x(s)-y(s)|^{p}ds\nonumber\\
&&+2^{p-1}(K+2(p-1)K^{2})\int_{0}^{t}\mathbb{E}|x([s/\tau]\tau)-x(s)|^{p}ds
\end{eqnarray}
In view of \Cref{lem32}, we have
\begin{eqnarray}\label{x_y_2_2}
&&2^{p-1}(K+2(p-1)K^{2})\int_{0}^{t}\mathbb{E}|x([s/\tau]\tau)-x(s)|^{p}ds\nonumber\\
&\le& 2^{p-1}(K+2(p-1)K^{2})\int_{0}^{t}C_{1}(K, p, \tau)\tau^{\frac{p}{2}}e^{2pK[1+(p-1)K]s}| x_{0}|^{p}ds\nonumber\\
&=&\frac{ 2^{p-1}(1+2(p-1)K)\tau^{\frac{p}{2}}C_{1}(K, p, \tau)|x_{0}|^{p}}{p[2+2(p-1)K]}\left(e^{2pK[1+(p-1)K]t}-1\right)\nonumber\\
&\le& \frac{2^{p-1}\tau^{\frac{p}{2}}C_{1}(K, p, \tau)|x_{0}|^{p}}{p}\left(e^{2pK[1+(p-1)K]t}-1\right)
\end{eqnarray}
Substituting \eqref{x_y_2_2} into \eqref{x_y_2} and using Gronwall inequality, we show that
\begin{equation*}
\mathbb{E}| x(t)-y(t)|^{p}\le C_{2}(K, p, \tau)\tau
^{\frac{p}{2}}|x_{0}|^{p}\left(e^{C_{3}(p, K)t}-1\right),
\end{equation*}
where
\begin{equation}\label{C_2}
C_{2}(K, p, \tau)=\frac{2^{p-1}C_{1}(K, p, \tau)}{p},
\end{equation}
\begin{equation}\label{C_3}
C_{3}(p, K)=[4p-1+2^{p-1}+2(p-1)(2p-1+2^{p-1})K]K.
\end{equation}
This completes the proof of \Cref{x_y_error_moment}.
\end{proof}

Our positive answer to (Q1) is stated in the following theorem.
\begin{theorem}\label{theorem 3}
Let \Cref{assumption 1} hold and the SDEPCA \eqref{SDEPCAs} is $p$th moment exponentially stable, i.e. $\mathbb{E}|x(t)|^{p}\le M_{1}e^{-\gamma_{1}t}|x_{0}|^{p}$. Choose $\delta\in(0, 1)$, if $\tau$ satisfies
\begin{equation}\label{Q1tau}
R(\tau)=\delta+2^{p-1}C_{2}(K, p, \tau)\tau
^{\frac{p}{2}}\left(e^{C_{3}(p, K)\left(\frac{\ln \left(\frac{2^{p-1}M_{1}}{\delta}\right)}{\gamma_{1}}+\tau\right)}-1\right)<1,
\end{equation}
then the SDE \eqref{SDEs} is also $p$th moment exponentially stable, where $C_{2}(K,p, \tau)$ and $C_{3}(p, K)$ are defined in \Cref{x_y_error_moment}.
\end{theorem}
\begin{proof}
{$\mathit{Step 1.}$} Let us choose a positive integer $\hat{n}$ such that 
\begin{equation*}
\frac{\ln\left( \frac{2^{p-1}M_{1}}{\delta}\right)}{\gamma_{1}\tau}\le \hat{n}<\frac{\ln \left(\frac{2^{p-1}M_{1}}{\delta}\right)}{\gamma_{1}\tau}+1.
\end{equation*}
So $2^{p-1}M_{1}e^{-\gamma_{1}\hat{n}\tau}\le\delta$.
Hence, 
\begin{equation}\label{y_(n+1)_1}
2^{p-1}\mathbb{E}| x(\hat{n}\tau)|^{p}\le 2^{p-1}M_{1}e^{-\gamma_{1}\hat{n}\tau}| x_{0}|^{p}\le \delta|x_{0}|^{p}.
\end{equation}
By virtue of \Cref{x_y_error_moment}, we obtain
\begin{equation*}\label{x_y_(n+1)_1}
\mathbb{E}| x(\hat{n}\tau)-y(\hat{n}\tau)|^{p}\le C_{2}(K, p, \tau)\tau
^{\frac{p}{2}}|x_{0}|^{p}\left(e^{C_{3}(p, K)\hat{n}\tau}-1\right),
\end{equation*}
which together with \eqref{y_(n+1)_1}, we arrive at
\begin{eqnarray*}
\mathbb{E}| y(\hat{n}\tau)|^{p}\le \left[\delta+2^{p-1}C_{2}(K, p, \tau)\tau
^{\frac{p}{2}}\left(e^{C_{3}(p, K)\hat{n}\tau}-1\right)\right]| x_{0}|^{p}\le R(\tau)| x_{0}|^{p}
\end{eqnarray*}
In view of \eqref{Q1tau}, there is a positive constant $\gamma_{2}$ such that $R(\tau)=e^{-\gamma_{2}\hat{n}\tau}$. Consequently, 
\begin{equation*}\label{hatn_1}
\mathbb{E}| y(\hat{n}\tau)|^{p}\le e^{-\gamma_{2}\hat{n}\tau}| x_{0}|^{p}.
\end{equation*}
{$\mathit{Step 2.}$} For any given $k\in \mathbb{N}^{+}$, let $\bar{x}(t)$ be the solution to the SDEPCA \eqref{SDEPCAs} for $t\ge k\hat{n}\tau$ with the initial value $\bar{x}(k\hat{n}\tau)=y(k\hat{n}\tau)$. We have from \eqref{x-markov} that 
\begin{equation}\label{3_x}
\mathbb{E}|\bar{x}((k+1)\hat{n}\tau)|^{p}\le M_{1}\mathbb{E}|y(k\hat{n}\tau)|^{p}e^{-\gamma_{1}\hat{n}\tau}.
\end{equation}
In view of \Cref{x_y_error_moment}, we arrive at
\begin{equation}\label{3_x-y}
\mathbb{E}|\bar{x}((k+1)\hat{n}\tau)-y((k+1)\hat{n}\tau)|^{p}\le C_{2}(K,p,\tau)\tau^{\frac{p}{2}}\mathbb{E}|y(k\hat{n}\tau)|^{p}\left(e^{C_{3}(p,K)\hat{n}\tau}-1\right).
\end{equation}
Using \eqref{3_x} and \eqref{3_x-y}, we can show, in the same way as we did in ${\mathit{Step 1}}$, that
\begin{equation*}
\mathbb{E}|y((k+1)\hat{n}\tau)|^{p}\le \mathbb{E}|y(k\hat{n}\tau)|^{p}e^{-\gamma_{2}\hat{n}\tau}.
\end{equation*}
Consequently,
\begin{eqnarray}\label{eq42}
\mathbb{E}|y(k\hat{n}\tau)|^{p}
\le e^{-\gamma_{2}\hat{n}\tau}\mathbb{E}|y((k-1)\hat{n}\tau)|^{p}\le\cdots\le e^{-k\gamma_{2}\hat{n}\tau}|x_{0}|^{p}
\end{eqnarray} 
Now, for any $t>0$, there is a unique $k$ such that $k\hat{n}\tau\le t<(k+1)\hat{n}\tau$.
In view of It\^{o} formula and \Cref{assumption 1}, similarly as the proof of  \Cref{lem21}, we arrive at
\begin{equation*}\label{ito_y}
\mathbb{E}|y(t)|^{p}\le \mathbb{E}|y(k\hat{n}\tau)|^{p}+2pK\left(1+(p-1)K\right)\int_{k\hat{n}\tau}^{t}\mathbb{E}|y(s)|^{p}ds.
\end{equation*}
By the Gronwall inequality and \eqref{eq42}, we can derive
\begin{eqnarray*}
\mathbb{E}|y(t)|^{p}&\le& \mathbb{E}|y(k\hat{n}\tau)|^{p}e^{2pK\left(1+(p-1)K\right)(t-k\hat{n}\tau)}\nonumber\\
&\le& \mathbb{E}|y(k\hat{n}\tau)|^{p}e^{2pK\left(1+(p-1)K\right)\hat{n}\tau}\nonumber\\
&\le& e^{-k\gamma_{2}\hat{n}\tau}|x_{0}|^{p}e^{2pK\left(1+(p-1)K\right)\hat{n}\tau}\nonumber\\
&\le& M_{2}|x_{0}|^{p}e^{-\gamma_{2} t},
\end{eqnarray*}
where $M_{2}=e^{\left[\gamma_{2}+2pK\left(1+(p-1)K\right)\right]\hat{n}\tau}$. The proof is hence complete.
\end{proof}

\section{EMSDEPCA \eqref{SDEPCAs-Euler-1} shares the stability with EMSDE \eqref{SDEs-Euler-1}}\label{sec 4}
In this section, we shall show that if the EMSDE \eqref{SDEs-Euler-1} is $p$th moment exponentially stable, then the EMSDEPCAs \eqref{SDEPCAs-Euler-1} is also $p$th moment exponentially stable, i.e. give the positive answer to (Q3). It is known from \Cref{rem1} that if $h=\tau$, then EMSDE \eqref{SDEs-Euler-1} and EMSDEPCA \eqref{SDEPCAs-Euler-1} are the same, and the answer for (Q3) is obviously positive. So in this section, we assume $h\neq \tau$.  
 
 \begin{theorem}\label{theorem 4}
 Assume that \Cref{assumption 1} holds. For a step size $h=\tau/m$, the EMSDE \eqref{SDEs-Euler-1} is $p$th moment exponentially stable , i.e. $\mathbb{E}|Y_{n}|^{p}\le L_{2}e^{-\lambda_{2}nh}|x_{0}|^{p}$. Choose $\delta\in(0, 1)$, if $\tau$ satisfies 
 \begin{equation}\label{H4tau}
 2^{p-1}H_{4}\left(2\left(\frac{\ln(2^{p-1}L_{2}/\delta)}{\lambda_{2}}+\tau\right), K, \tau, p\right)\tau^{\frac{p}{2}}+\delta<1
 \end{equation}
 where $H_{4}(T, K,\tau, p)$ is defined in \Cref{lem53}, then the EMSDEPCA \eqref{SDEPCAs-Euler-1} is also $p$th moment exponentially stable.
 \end{theorem}

 The above theorem will be proved below by making use of the following lemmas.
 \begin{lemma}\label{lem23}
Assume that \Cref{assumption 1} holds. Then for any given $T>0$ such that
\begin{equation*}
\sup_{0\le t_{n}\le T}\mathbb{E}|X_{n}|^{p}\le H_{3}(T, p, K)|x_{0}|^{p},
\end{equation*}
where $H_{3}(T, p, K)=e^{2pK\left(1+(p-1)K\right)T}$.
\end{lemma}
\begin{proof}
The proof follows from \Cref{lem21}. But to highlight the importance of numerical solutions, it is given here. In view of It\^{o} formula and \Cref{assumption 1}, we have
\begin{eqnarray*}
\mathbb{E}|x_{\Delta}(v)|^{p}&=&|x_{0}|^{p}+\mathbb{E}\int_{0}^{t}p|x_{\Delta}(s)|^{p-2}x_{\Delta}(s)^{T}(f(\bar{x}_{\Delta}(s))+u_{1}(\bar{x}_{\Delta}([s/\tau]\tau)))\nonumber\\
&&+\frac{p(p-1)}{2}|x_{\Delta}(s)|^{p-2}|g(\bar{x}_{\Delta}(s))+u_{2}(\bar{x}_{\Delta}([s/\tau]\tau))|^{2}ds\\
&\le& |x_{0}|^{p}+\mathbb{E}\int_{0}^{v}pK|x_{\Delta}(s)|^{p-1}(|\bar{x}_{\Delta}(s)|+|\bar{x}_{\Delta}([s/\tau]\tau)|)ds\nonumber\\
&&+p(p-1)K^{2}\mathbb{E}\int_{0}^{v}|x_{\Delta}(s)|^{p-2}\left(|\bar{x}_{\Delta}(s)|^{2}+|\bar{x}_{\Delta}([s/\tau]\tau)|^{2}\right)ds\nonumber\\
&\le& |x_{0}|^{p}+2pK\left(1+(p-1)K\right)\int_{0}^{v}\sup_{0\le u\le s}\mathbb{E}|x_{\Delta}(u)|^{p}ds
\end{eqnarray*}
According to the Gronwall inequality, we obtain
\begin{equation}\label{eq27}
\sup_{0\le t\le T}\mathbb{E}|x_{\Delta}(t)|^{p}\le |x_{0}|^{p}e^{2pK\left(1+(p-1)K\right)T}.
\end{equation}
The proof is completed by noting that $x_{\Delta}(t_{n})=X_{n}$, i.e.
\begin{equation*}
\sup_{0\le t_{n}\le T}\mathbb{E}|X_{n}|^{p}\le |x_{0}|^{p}e^{2pK\left(1+(p-1)K\right)T}.
\end{equation*}
\end{proof}
 The following lemma estimates the difference in the $p$th moment between approximation of EMSDE \eqref{SDEs-Euler-1} and that of EMSDEPCA \eqref{SDEPCAs-Euler-1}.
\begin{lemma}\label{lem53}
Let \Cref{assumption 1} hold. Then for any given positive constant $T>0$, 
\begin{eqnarray*}
\sup_{0\le t_{n}\le T}\mathbb{E}|X_{n}-Y_{n}|^{p}\le H_{4}(T, K, \tau, p)\tau^{\frac{p}{2}}|x_{0}|^{p},
\end{eqnarray*}
where $H_{4}(T, K, \tau, p)=C_{2}(K, p, \tau)\left(e^{C_{3}(p, K)T}-1\right)$, $C_{2}$ and $C_{3}$ are defined in \Cref{x_y_error_moment}.
\end{lemma}
\begin{proof}
According to \eqref{eq_x_delta}, \eqref{eq_y_delta}, It\^{o} formula and \Cref{assumption 1}, we have
\begin{eqnarray}\label{eq3}
&&\mathbb{E}|x_{\Delta}(v)-y_{\Delta}(v)|^{p}\nonumber\\
&\le& \mathbb{E}\int_{0}^{v}p|x_{\Delta}(s)-y_{\Delta}(s)|^{p-1}|f(\bar{x}_{\Delta}(s))-f(\bar{y}_{\Delta}(s))+u_{1}(\bar{x}_{\Delta}([s/\tau]\tau))-u_{1}(\bar{y}_{\Delta}(s))|\nonumber\\
&&+ \frac{p(p-1)}{2}|x_{\Delta}(s)-y_{\Delta}(s)|^{p-2}|g(\bar{x}_{\Delta}(s))-g(\bar{y}_{\Delta}(s))+u_{2}(\bar{x}_{\Delta}([s/\tau]\tau))-u_{2}(\bar{y}_{\Delta}(s))|^{2}ds\nonumber\\
&\le& \mathbb{E}\int_{0}^{v}pK|x_{\Delta}(s)-y_{\Delta}(s)|^{p-1}(|\bar{x}_{\Delta}(s)-\bar{y}_{\Delta}(s)|+|\bar{x}_{\Delta}([s/\tau]\tau)-\bar{y}_{\Delta}(s)|)\nonumber\\
&&+ p(p-1)K^{2}|x_{\Delta}(s)-y_{\Delta}(s)|^{p-2}(|\bar{x}_{\Delta}(s)-\bar{y}_{\Delta}(s)|^{2}+|\bar{x}_{\Delta}([s/\tau]\tau)-\bar{y}_{\Delta}(s)|^{2})ds\nonumber\\
&\le& (pK+p(p-1)K^{2})\int_{0}^{v}\sup_{0\le u\le s}\mathbb{E}|x_{\Delta}(u)-y_{\Delta}(u)|^{p}ds\nonumber\\
&&+ \mathbb{E}\int_{0}^{v}pK|x_{\Delta}(s)-y_{\Delta}(s)|^{p-1}|\bar{x}_{\Delta}([s/\tau]\tau)-\bar{y}_{\Delta}(s)|ds\nonumber\\
&&+ \mathbb{E}\int_{0}^{v}p(p-1)K^{2}|x_{\Delta}(s)-y_{\Delta}(s)|^{p-2}|\bar{x}_{\Delta}([s/\tau]\tau)-\bar{y}_{\Delta}(s)|^{2}ds\nonumber\\
\end{eqnarray}
Similarly as in the proof of \Cref{lem32}, we obtain
\begin{eqnarray}\label{eq2}
\mathbb{E}|\bar{x}_{\Delta}(t)-\bar{x}_{\Delta}([t/\tau]\tau)|^{p}&=&\mathbb{E}|X_{km+l}-X_{km}|^{p}\nonumber\\
&\le& C_{1}(K, p, \tau)\tau^{\frac{p}{2}}|x_{0}|^{p}e^{2pK\left(1+(p-1)K\right)t}
\end{eqnarray}
Substituting \eqref{eq2} into \eqref{eq3}, we have
\begin{eqnarray*}
&&\mathbb{E}|x_{\Delta}(v)-y_{\Delta}(v)|^{2}\nonumber\\
&\le& \left[(2p-1+2^{p-1})+2(p-1)(p-1+2^{p-1})K\right]K\int_{0}^{v}\sup_{0\le u\le s}\mathbb{E}|x_{\Delta}(u)-y_{\Delta}(u)|^{p}ds\nonumber\\
&&+2^{p-1}\left(1+2(p-1)K\right)K\int_{0}^{v}C_{1}(K, p, \tau)\tau^{\frac{p}{2}}|x_{0}|^{p}e^{2pK\left(1+(p-1)K\right)t}ds
\end{eqnarray*}
Applying the Gronwall inequality, we have
\begin{eqnarray*}\label{eq55}
\sup_{0\le v\le T}\mathbb{E}|x_{\Delta}(v)-y_{\Delta}(v)|^{p}\le C_{2}(K, p, \tau)\tau^{\frac{p}{2}}|x_{0}|^{p}\left(e^{C_{3}(p, K)T}-1\right),
\end{eqnarray*}
where $C_{2}(K, p, \tau)$ and $C_{3}(p, K)$ are defined in Lemma \ref{x_y_error_moment}. For ease of notations, set $H_{4}(T, K, p, \tau)=C_{2}(K, p, \tau)\left(e^{C_{3}(p, K)T}-1\right)$.
The proof is completed by noting that $x_{\Delta}(t_{n})=X_{n}$ and $y_{\Delta}(t_{n})=Y_{n}$, i.e.
\begin{eqnarray*}
\sup_{0\le t_{n}\le T}\mathbb{E}|X_{n}-Y_{n}|^{p}\le H_{4}(T, K, p, \tau)\tau^{\frac{p}{2}}|x_{0}|^{p}.
\end{eqnarray*}
\end{proof}

\begin{proof}[The proof of \Cref{theorem 4}]
Let
\begin{equation*}
\hat{n}=\left[\frac{\ln(\frac{2^{p-1}L_{2}}{\delta})}{\lambda_{2}\tau}\right]+1,
\end{equation*}
which implies that
\begin{equation*}\label{eq56}
2^{p-1}L_{2}e^{-\lambda_{2} \hat{n}\tau}\le \delta,\quad and\quad
\hat{n}\tau\le \frac{\ln(\frac{2^{p-1}L_{2}}{\delta})}{\lambda_{2}}+\tau.
\end{equation*}
By $|a+b|^{p}\le 2^{p-1}|a|^{p}+2^{p-1}|b|^{p}$, we have
\begin{equation*}
\mathbb{E}|X_{n}|^{p}\le 2^{p-1}\mathbb{E}|X_{n}-Y_{n}|^{p}+2^{p-1}\mathbb{E}|Y_{n}|^{p}.
\end{equation*}
According to the $p$th moment exponentially stability of EMSDE \eqref{SDEs-Euler-1} and \Cref{lem53}, we have
\begin{eqnarray*}
\sup_{\hat{n}\tau\le t_{n}\le 2\hat{n}\tau}\mathbb{E}|X_{n}|^{p}
&\le& 2^{p-1}\sup_{0\le t_{n}\le 2\hat{n}\tau}\mathbb{E}|X_{n}-Y_{n}|^{p}+2^{p-1}\sup_{\hat{n}\tau\le t_{n}\le 2\hat{n}\tau}\mathbb{E}|Y_{n}|^{p}\nonumber\\
&\le& \left(2^{p-1}H_{4}(2\hat{n}\tau, K, \tau, p)\tau^{\frac{p}{2}}+2^{p-1}L_{2}e^{-\lambda_{2}\hat{n}\tau}\right)|x_{0}|^{p}\nonumber\\
&\le& \left(2^{p-1}H_{4}(2\hat{n}\tau, K, \tau, p)\tau^{\frac{p}{2}}+\delta\right)|x_{0}|^{p}\nonumber\\
&\le& \left(2^{p-1}H_{4}\left(2\left(\frac{\ln(2^{p-1}L_{2}/\delta)}{\lambda_{2}}+\tau\right), K, \tau, p\right)\tau^{\frac{p}{2}}+\delta\right)|x_{0}|^{p}
\end{eqnarray*}
Let $R(\tau)=2^{p-1}H_{4}\left(2\left(\frac{\ln(2^{p-1}L_{2}/\delta)}{\lambda_{2}}+\tau\right), K, \tau, p\right)\tau^{\frac{p}{2}}+\delta$. 
It is known from \eqref{H4tau} that $R(\tau)<1$. Therefore, we can find a positive constant $\lambda_{1}$ such that 
\begin{equation}\label{eq57*}
R(\tau)<e^{-\lambda_{1}\hat{n}\tau},
\end{equation}
and
\begin{equation}\label{eq57}
\sup_{\hat{n}\tau\le t_{n}\le 2\hat{n}\tau}\mathbb{E}|X_{n}|^{p}\le e^{-\lambda_{1}\hat{n}\tau}|x_{0}|^{p}.
\end{equation}
Let $\{\bar{Y}_{n}\}_{t_{n}\ge \hat{n}\tau}$ be the solution of EMSDE \eqref{SDEs-Euler-1} with initial data $\bar{Y}_{\hat{n}m}=X_{\hat{n}m}$ at initial time $t=\hat{n}\tau$. According to \Cref{lem53}, we have
\begin{equation*}
\sup_{\hat{n}\tau\le t_{n}\le 3\hat{n}\tau}\mathbb{E}|X_{n}-\bar{Y}_{n}|^{p}\le H_{4}\left(2\left(\frac{\ln(2^{p-1}L_{2}/\delta)}{\lambda_{2}}+\tau\right), K, \tau, p\right)\tau^{\frac{p}{2}}\mathbb{E}|X_{\hat{n}m}|^{p}.
\end{equation*}
It comes from \eqref{Y-markov} that
\begin{equation*}
\mathbb{E}|\bar{Y}_{n}|^{p}\le L_{2}e^{-\lambda_{2}(nh-\hat{n}mh)}\mathbb{E}|X_{\hat{n}m}|^{p}.
\end{equation*}
Using similar arguments that produced \eqref{eq57}, we obtain
\begin{eqnarray*}
\sup_{2\hat{n}\tau\le t_{n}\le 3\hat{n}\tau}\mathbb{E}|X_{n}|^{p}&\le& R(\tau)\mathbb{E}|X_{\hat{n}m}|^{p}\le e^{-\lambda_{1}\hat{n}\tau}\mathbb{E}|X_{\hat{n}m}|^{p}\le e^{-\lambda_{1}\hat{n}\tau}\sup_{\hat{n}\tau\le t_{n}\le 2\hat{n}\tau}\mathbb{E}|X_{n}|^{p}
\end{eqnarray*}
By \eqref{eq57}, we obtain
\begin{equation*}
\sup_{2\hat{n}\tau\le t_{n}\le 3\hat{n}\tau}\mathbb{E}|X_{n}|^{p}\le e^{-2\lambda_{1}\hat{n}\tau}|x_{0}|^{p}
\end{equation*}
 Continuing this approach and using \eqref{eq57*}, we have, for any $i=1, 2, \cdots$,
 \begin{eqnarray}\label{eq61}
 \sup_{i\hat{n}\tau\le t_{n}\le (i+1)\hat{n}\tau}\mathbb{E}|X_{n}|^{p}\le e^{-i\lambda_{1}\hat{n}\tau}|x_{0}|^{p}\le \bar{L}_{1}e^{-\lambda_{1}nh}|x_{0}|^{p}
 \end{eqnarray}
 where $\bar{L}_{1}=e^{\lambda_{1}\hat{n}\tau}$.
For $i=0$, by using \Cref{lem23}, we get
\begin{eqnarray*}
\sup_{0\le t_{n}\le \hat{n}\tau}\mathbb{E}|X_{n}|^{p}\le H_{3}(\hat{n}\tau, p, K)|x_{0}|^{p}\le L_{1}|x_{0}|^{p}e^{-\lambda_{1}nh},
\end{eqnarray*}
where $L_{1}=H_{3}(\hat{n}\tau, p, K)e^{\lambda_{1}\hat{n}\tau}>e^{\lambda_{1}\hat{n}\tau}=\bar{L}_{1}$. This, together with \eqref{eq61}, we arrive at for all $n\in\mathbb{N}$
\begin{equation*}
\mathbb{E}|X_{n}|^{p}\le L_{1}|x_{0}|^{p}e^{-\lambda_{1}nh}.
\end{equation*}
\end{proof}

\section{SDEPCA \eqref{SDEPCAs} shares the stability with EMSDEPCA \eqref{SDEPCAs-Euler-1}}\label{sec 5}
In this section, we shall show that for a given step size $h$, if the EMSDEPCA \eqref{SDEPCAs-Euler-1} is $p$th moment exponentially stable, then the SDEPCA \eqref{SDEPCAs} is also $p$th moment exponentially stable with some restriction with $h$, i.e. give the positive answer to (Q4). The first lemma shows that the EMSDEPCA \eqref{SDEPCAs-Euler-1} is convergent in the $p$th moment to SDEPCA \eqref{SDEPCAs}.
\begin{lemma}\label{lem73}
Assume that Assumption \ref{assumption 1} holds. For $T>0$, 
\begin{equation*}
\sup_{0\le t_{n}\le T}\mathbb{E}|x(t_{n})-X_{n}|^{p}\le H_{6}(T, K, p)h^{\frac{p}{2}}|x_{0}|^{p},
\end{equation*}
where $H_{6}(T, K, p)$ is defined as \eqref{H_6}.
\end{lemma}
\begin{proof}
For any $t\ge 0$, by It\^{o} formula, Assumption \ref{assumption 1} and Young inequality, we obtain
\begin{eqnarray*}
&&\mathbb{E}|x(t)-x_{\Delta}(t)|^{p}\nonumber\\
&\le& \mathbb{E}\int_{0}^{t}pK|x(s)-x_{\Delta}(s)|^{p-1}(|x(s)-\bar{x}_{\Delta}(s)|+|x([s/\tau]\tau))-\bar{x}_{\Delta}([s/\tau]\tau)|)\nonumber\\
&&+p(p-1)K^{2}|x(s)-x_{\Delta}(s)|^{p-2}\left(|x(s)-\bar{x}_{\Delta}(s)|^{2}+|x([s/\tau]\tau)-\bar{x}_{\Delta}([s/\tau]\tau)|^{2}\right)ds\nonumber\\
&\le& 2pK(1+2(p-1)K)\int_{0}^{t}\mathbb{E}\sup_{0\le u\le s}|x(s)-x_{\Delta}(s)|^{p}ds\nonumber\\
&&+pK\mathbb{E}\int_{0}^{t}|x(s)-x_{\Delta}(s)|^{p-1}|x_{\Delta}(s)-\bar{x}_{\Delta}(s)|ds\nonumber\\
&&+pK\mathbb{E}\int_{0}^{t}|x(s)-x_{\Delta}(s)|^{p-1}|x_{\Delta}([s/\tau]\tau)-\bar{x}_{\Delta}([s/\tau]\tau)|ds\nonumber\\
&&+2p(p-1)K^{2}\mathbb{E}\int_{0}^{t}|x(s)-x_{\Delta}(s)|^{p-2}|x_{\Delta}(s)-\bar{x}_{\Delta}(s)|^{2}ds\nonumber\\
&&+2p(p-1)K^{2}\mathbb{E}\int_{0}^{t}|x(s)-x_{\Delta}(s)|^{p-2}|x_{\Delta}([s/\tau]\tau)-\bar{x}_{\Delta}([s/\tau]\tau)|^{2})ds
\end{eqnarray*}
By noting $x_{\Delta}([s/\tau]\tau)-\bar{x}_{\Delta}([s/\tau]\tau)=0$, we have
\begin{eqnarray}\label{eq75}
\mathbb{E}|x(t)-x_{\Delta}(t)|^{p}
&\le& K[3p-1+2(p-1)(3p-2)K]\int_{0}^{t}\mathbb{E}\sup_{0\le u\le s}|x(u)-x_{\Delta}(u)|^{p}ds\nonumber\\
&&+K[1+4(p-1)K]\int_{0}^{t}\mathbb{E}|x_{\Delta}(s)-\bar{x}_{\Delta}(s)|^{p}ds
\end{eqnarray}
Now, we shall give the estimation of the second term of the right hand. For any $t>0$, there exists $k$ and $l$ such that $t_{km+l}\le t< t_{km+l+1}$. Then $\bar{x}_{\Delta}(t)=X_{km+l}=x_{\Delta}(t_{km+l})$. Hence from \eqref{SDEPCAs-Euler-2} we have
\begin{eqnarray*}
&&\mathbb{E}|x_{\Delta}(t)-\bar{x}_{\Delta}(t)|^{p}\nonumber\\
&=&\mathbb{E}\left|(t-t_{km+l})\left(f(X_{km+l})+u_{1}(X_{km})\right)+\left(g(X_{km+l})+u_{2}(X_{km})\right)\left(w(t)-w(t_{km+l})\right)\right|^{p}\nonumber\\
&\le& 2^{2p-1}h^{\frac{p}{2}}K^{p}\left(\mathbb{E}|X_{km+l}|^{p}+\mathbb{E}|X_{km}|^{p}\right)
\end{eqnarray*}
Applying \eqref{eq27}, we obtain
\begin{eqnarray}\label{eq77}
\mathbb{E}|x_{\Delta}(t)-\bar{x}_{\Delta}(t)|^{p}&\le& 2^{2p}h^{\frac{p}{2}}K^{p}|x_{0}|^{p}e^{2pK\left(1+(p-1)K\right)t}
\end{eqnarray}
Substituting \eqref{eq77} into \eqref{eq75}, we obtain
\begin{eqnarray*}
&&\mathbb{E}|x(v)-x_{\Delta}(v)|^{p}\nonumber\\
&\le&  K[3p-1+2(p-1)(3p-2)K]\int_{0}^{t}\mathbb{E}\sup_{0\le u\le s}|x(u)-x_{\Delta}(u)|^{p}ds\nonumber\\
&&+K[1+4(p-1)K]\int_{0}^{t}2^{2p}h^{\frac{p}{2}}K^{p}|x_{0}|^{p}e^{2pK\left(1+(p-1)K\right)s}ds
\end{eqnarray*}
By Gronwall inequality, we have
\begin{equation*}
\sup_{0\le t\le T}\mathbb{E}|x(t)-x_{\Delta}(t)|^{p}\le H_{6}(T, p, K)h^{\frac{p}{2}}|x_{0}|^{p},
\end{equation*}
where 
\begin{equation}\label{H_6}
H_{6}(T, p, K)=[1+4(p-1)K]2^{2p}K^{p+1}e^{KT[5p-1+4(p-1)(2p-1)K]}T.
\end{equation}
By noting $x_{\Delta}(t_{n})=X_{n}$, we get for $t=t_{n}$
\begin{equation*}
\sup_{0\le t_{n}\le T}\mathbb{E}|x(t_{n})-X_{n}|^{p}\le H_{6}(T, p, K)h^{\frac{p}{2}}|x_{0}|^{p}
\end{equation*}
The proof is completed.
\end{proof}

\begin{lemma}\label{lem22}
Assume that \Cref{assumption 1} holds. Then for any $0\le t_{n}\le t\le t_{n+1}\le T$, 
\begin{equation*}
\sup_{0\le t\le T}\mathbb{E}|x(t)-x(t_{n})|^{p}\le H_{7}(T, K, p)h^{\frac{p}{2}}|x_{0}|^{p},
\end{equation*}
where $H_{7}(T, K, p)=2^{2p-1}K^{p}\left[T^{\frac{p}{2}}+\left(\frac{p(p-1)}{2}\right)^{\frac{p}{2}}\right]e^{2pK[1+(p-1)K]T}$.
\end{lemma}
\begin{proof}
For any $0\le t_{n}\le t\le t_{n+1}\le T$, we have
\begin{eqnarray*}
x(t)-x(t_{n})&=&\int_{t_{n}}^{t}f(x(s))+u_{1}(x([s/\tau]\tau))ds+\int_{t_{n}}^{t}g(x(s))+u_{2}(x([s/\tau]\tau))dw(s).
\end{eqnarray*}
In view of H$\ddot{o}$lder inequality, \Cref{assumption 1} as well as moment inequality, we have
\begin{eqnarray*}
&&\mathbb{E}|x(t)-x(t_{n})|^{p}\nonumber\\
&\le& 2^{2p-2}(t-t_{n})^{\frac{p}{2}-1}K^{p}\left[(t-t_{n})^{\frac{p}{2}}+\left(\frac{p(p-1)}{2}\right)^{\frac{p}{2}}\right]\int_{t_{n}}^{t}\mathbb{E}|x(s)|^{p}+\mathbb{E}|x([s/\tau]\tau)|^{p}ds\nonumber\\
&\le& 2^{2p-1}(t-t_{n})^{\frac{p}{2}-1}K^{p}\left[(t-t_{n})^{\frac{p}{2}}+\left(\frac{p(p-1)}{2}\right)^{\frac{p}{2}}\right]\int_{t_{n}}^{t}\sup_{0\le u\le s}\mathbb{E}|x(u)|^{p}ds
\end{eqnarray*}
It follows from \eqref{eq21} that
\begin{eqnarray*}
\mathbb{E}|x(t)-x(t_{n})|^{p}&\le&  2^{2p-1}(t-t_{n})^{\frac{p}{2}}K^{p}\left[(t-t_{n})^{\frac{p}{2}}+\left(\frac{p(p-1)}{2}\right)^{\frac{p}{2}}\right]e^{2pK[1+(p-1)K]t}| x_{0}|^{p}\nonumber\\
&\le& 2^{2p-1}h^{\frac{p}{2}}K^{p}\left[t^{\frac{p}{2}}+\left(\frac{p(p-1)}{2}\right)^{\frac{p}{2}}\right]e^{2pK[1+(p-1)K]t}| x_{0}|^{p}
\end{eqnarray*}
Hence,
\begin{equation*}
\sup_{0\le t\le T}\mathbb{E}|x(t)-x(t_{n})|^{p}\le 2^{2p-1}K^{p}\left[T^{\frac{p}{2}}+\left(\frac{p(p-1)}{2}\right)^{\frac{p}{2}}\right]e^{2pK[1+(p-1)K]T}h^{\frac{p}{2}}| x_{0}|^{p}.
\end{equation*}
The proof is complete.
\end{proof}

\begin{theorem}\label{Q6}
Assume that \Cref{assumption 1} holds. For a step size $h=\frac{\tau}{m}$, the EMSDEPCA \eqref{SDEPCAs-Euler-1} is $p$th moment exponentially stable, i.e. $\mathbb{E}|X_{n}|^{p}\le L_{1}e^{-\lambda_{1}nh}|x_{0}|^{p}$. If the step size $h$ satisfies  
\begin{equation}\label{eq81}
3^{p-1}H_{8}(2\hat{n}\tau, K, p)h^{\frac{p}{2}}+e^{-\frac{3}{4}\lambda_{1}\hat{n}\tau}\le e^{-\frac{1}{2}\lambda_{1}\hat{n}\tau},
\end{equation}
where $\hat{n}=\left[\frac{4\ln( 3^{p-1}L_{1})}{\lambda_{1}\tau}\right]+1$ and $H_{8}(2\hat{n}\tau, K, p)=H_{7}(2\hat{n}\tau, K, p)+H_{6}(2\hat{n}\tau, K, p)$, then the SDEPCA \eqref{SDEPCAs} is also $p$th moment exponentially stable, where $H_{6}(2\hat{n}\tau, K, p)$ is defined in \Cref{lem73} and $H_{7}(2\hat{n}\tau, K, p)$ in \Cref{lem22}.
\end{theorem}
\begin{proof}
For any $t\ge 0$, there exist $n\in\mathbb{N}$ such that $t_{n}\le t<t_{n+1}$,
\begin{eqnarray*}
\mathbb{E}|x(t)|^{p}\le 3^{p-1}\mathbb{E}|x(t)-x(t_{n})|^{p}+3^{p-1}\mathbb{E}|x(t_{n})-X_{n}|^{p}+3^{p-1}\mathbb{E}|X_{n}|^{p}
\end{eqnarray*}
According to \Cref{lem73}, we have
\begin{equation*}
\sup_{0\le t_{n}\le 2\hat{n}\tau}\mathbb{E}|x(t_{n})-X_{n}|^{p}\le H_{6}(2\hat{n}\tau, K, p)h^{\frac{p}{2}}|x_{0}|^{p}.
\end{equation*}
By \Cref{lem22}, we have
\begin{equation*}
\sup_{0\le t\le 2\hat{n}\tau}\mathbb{E}|x(t)-x(t_{n})|^{p}\le H_{7}(2\hat{n}\tau, K, p)h^{\frac{p}{2}}|x_{0}|^{p}.
\end{equation*}
Since $\hat{n}=\left[\frac{4\ln (3^{p-1}L_{1})}{\lambda_{1}\tau}\right]+1$, we have $3^{p-1}L_{1}e^{-\lambda_{1}\hat{n}\tau}\le e^{-\frac{3}{4}\lambda_{1}\hat{n}\tau}$, Therefore,
\begin{eqnarray*}
&&\sup_{\hat{n}\tau\le t\le 2\hat{n}\tau}\mathbb{E}|x(t)|^{p}\nonumber\\
&\le& 3^{p-1}\sup_{0\le t\le 2\hat{n}\tau}\mathbb{E}|x(t)-x(t_{n})|^{p}+3^{p-1}\sup_{0\le t_{n}\le 2\hat{n}\tau}\mathbb{E}|x(t_{n})-X_{n}|^{p}+3^{p-1}\sup_{\hat{n}\tau\le t_{n}\le 2\hat{n}\tau}\mathbb{E}|X_{n}|^{p}\nonumber\\
&\le& \left(3^{p-1}H_{7}(2\hat{n}\tau, K, p)h^{\frac{p}{2}}+3^{p-1}H_{6}(2\hat{n}\tau, K, p)h^{\frac{p}{2}}+3^{p-1}L_{1}e^{-\lambda_{1}\hat{n}\tau}\right)|x_{0}|^{p}\nonumber\\
&\le& \left(3^{p-1}H_{8}(2\hat{n}\tau, K, p)h^{\frac{p}{2}}+e^{-\frac{3}{4}\lambda_{1}\hat{n}\tau}\right)|x_{0}|^{p}
\end{eqnarray*}
where $H_{8}(2\hat{n}\tau, K, p)=H_{7}(2\hat{n}\tau, K, p)+H_{6}(2\hat{n}\tau, K, p)$. Recalling \eqref{eq81}, we have
\begin{eqnarray}\label{eq87}
\sup_{\hat{n}\tau\le t\le 2\hat{n}\tau}\mathbb{E}|x(t)|^{p}\le e^{-\frac{1}{2}\lambda_{1}\hat{n}\tau}|x_{0}|^{p}.
\end{eqnarray}
Denote by $\{\bar{X}_{n}\}_{nh\ge \hat{n}\tau}$ the numerical solution of \eqref{SDEPCAs-Euler-1} with initial data $\bar{X}_{\hat{n}m}=x(\hat{n}\tau)$ at $t=\hat{n}\tau$. Then from \eqref{X_markov}, we have
\begin{equation*}
\mathbb{E}|\bar{X}_{n}|^{p}\le L_{1}e^{-\lambda_{1}(n-\hat{n}m)h}\mathbb{E}|x(\hat{n}\tau)|^{p}.
\end{equation*}
Using \Cref{lem73} and \Cref{lem22}, we get
\begin{equation*}
\sup_{\hat{n}\tau\le t_{n}\le 3\hat{n}\tau}\mathbb{E}|x(t_{n})-\bar{X}_{n}|^{p}\le H_{6}(2\hat{n}\tau, K, p)h^{\frac{p}{2}}\mathbb{E}|x(\hat{n}\tau)|^{p}.
\end{equation*}
\begin{equation*}
\sup_{\hat{n}\tau\le t\le 3\hat{n}\tau}\mathbb{E}|x(t)-x(t_{n})|^{p}\le H_{7}(2\hat{n}\tau, K, p)h^{\frac{p}{2}}\mathbb{E}|x(\hat{n}\tau)|^{p}.
\end{equation*}
Therefore, 
\begin{eqnarray*}
\sup_{2\hat{n}\tau\le t\le 3\hat{n}\tau}\mathbb{E}|x(t)|^{p}&\le& \left(3^{p-1}H_{7}(2\hat{n}\tau, K, p)h^{\frac{p}{2}}+3^{p-1}H_{6}(2\hat{n}\tau, K, p)h^{\frac{p}{2}}+L_{1}e^{-\lambda_{1}\hat{n}\tau}\right)\mathbb{E}|x(\hat{n}\tau)|^{p}\nonumber\\
&\le& \left(3^{p-1}H_{8}(2\hat{n}\tau, K, p)h^{\frac{p}{2}}+e^{-\frac{3}{4}\lambda_{1}\hat{n}\tau}\right)\mathbb{E}|x(\hat{n}\tau)|^{p}\nonumber\\
&\le& e^{-\frac{1}{2}\lambda_{1}\hat{n}\tau}\sup_{\hat{n}\tau\le t\le 2\hat{n}\tau}\mathbb{E}|x(t)|^{p}
\end{eqnarray*}
By \eqref{eq87}, we obtain
\begin{equation*}
\sup_{2\hat{n}\tau\le t\le 3\hat{n}\tau}\mathbb{E}|x(t)|^{p}\le e^{-\frac{\lambda_{1}}{2}2\hat{n}\tau}|x_{0}|^{p}.
\end{equation*}
Repeating this procedure, we find for $i=1, 2, \cdots$, 
\begin{eqnarray}\label{eq87_1}
\sup_{i\hat{n}\tau\le t\le (i+1)\hat{n}\tau}\mathbb{E}|x(t)|^{p}\le e^{-\frac{\lambda_{1}}{2} i\hat{n}\tau}| x_{0}|^{p}\le \bar{M}_{1}e^{-\frac{\lambda_{1}}{2}t}|x_{0}|^{p},
\end{eqnarray}
where $\bar{M}_{1}=e^{\frac{\lambda_{1}}{2}\hat{n}\tau}$.  On the other hand, by means of \Cref{lem21}, we can show that 
\begin{eqnarray*}
\sup_{0\le t\le \hat{n}\tau}\mathbb{E}|x(t)|^{p}\le H_{1}(\hat{n}\tau, p, K)|x_{0}|^{p}\le M_{1}|x_{0}|^{p}e^{-\frac{\lambda_{1}}{2}t},
\end{eqnarray*}
where $M_{1}=H_{1}(\hat{n}\tau, p, K)e^{\frac{\lambda_{1}}{2}\hat{n}\tau}>e^{\frac{\lambda_{1}}{2}\hat{n}\tau}=\bar{M}_{1}$, this, together with \eqref{eq87_1}, we arrive at for any $t\ge 0$,
\begin{equation*}
\mathbb{E}|x(t)|^{p}\le M_{1}|x_{0}|^{p}e^{-\frac{1}{2}\lambda_{1}t}.
\end{equation*}
This completes the proof.
\end{proof}

\section{EMSDE \eqref{SDEs-Euler-1} shares the stability with SDE \eqref{SDEs}}\label{sec 6}
\cite{Mao2003LMS} gives the positive answer to (Q2) only for the case $p=2$. In this section, we shall show that for $p>2$, if the SDE \eqref{SDEs} is $p$th moment exponentially stable, then the EMSDE \eqref{SDEs-Euler-1} is also $p$th moment exponentially stable with some restriction on $h$, i.e. give the positive answer to (Q2). The first lemma shows that the EMSDE \eqref{SDEs-Euler-1} is convergent in the $p$th moment to SDE \eqref{SDEs}.
\begin{lemma}\label{lemA1}
Assume that \Cref{assumption 1} holds. For any $T>0$,
\begin{equation*}
\sup_{0\le t_{n}\le T}\mathbb{E}|y(t_{n})-Y_{n}|^{p}\le H_{9}(T, K, p)h^{\frac{p}{2}}|x_{0}|^{p},
\end{equation*}
where $H_{9}(T, K, p)$ is defined as \eqref{H9}.
\end{lemma}
\begin{proof}
For any $t\ge 0$, by It\^{o} formula, \Cref{assumption 1} and Young inequality, we obtain
\begin{eqnarray}\label{eqA107}
&&\mathbb{E}|y(t)-y_{\Delta}(t)|^{p}\nonumber\\
&\le&\mathbb{E}\int_{0}^{t}p|y(s)-y_{\Delta}(s)|^{p-1}\left|f(y(s))-f(\bar{y}_{\Delta}(s))+u_{1}(y(s))-u_{1}(\bar{y}_{\Delta}(s))\right|\nonumber\\
&&+\frac{p(p-1)}{2}|y(s)-y_{\Delta}(s)|^{p-2}|g(y(s))-g(\bar{y}_{\Delta}(s))+u_{2}(y(s))-u_{2}(\bar{y}_{\Delta}(s))|^{2}ds\nonumber\\
&\le& \mathbb{E}\int_{0}^{t}2pK|y(s)-y_{\Delta}(s)|^{p-1}|y(s)-\bar{y}_{\Delta}(s)|\nonumber\\
&&+2p(p-1)K^{2}|y(s)-y_{\Delta}(s)|^{p-2}|y(s)-\bar{y}_{\Delta}(s)|^{2}ds\nonumber\\
&\le& \left(2K(2p-1)+8(p-1)^{2}K^{2}\right)\int_{0}^{t}\mathbb{E}|y(s)-y_{\Delta}(s)|^{p}ds\nonumber\\
&&+\left(2K+8(p-1)K^{2}\right)\int_{0}^{t}\mathbb{E}|y_{\Delta}(s)-\bar{y}_{\Delta}(s)|^{p}ds
\end{eqnarray}
Now, we shall give the estimation of the second term of the right hand. For any $t>0$, there exists $n$ such that $t_{n}\le t< t_{n+1}$, and $\bar{y}_{\Delta}(t)=Y_{n}=y_{\Delta}(t_{n})$. Hence from \eqref{SDEs-Euler-1} we have
\begin{eqnarray*}
&&\mathbb{E}|y_{\Delta}(t)-\bar{y}_{\Delta}(t)|^{p}\nonumber\\
&=&\mathbb{E}\left|(t-t_{n})\left(f(Y_{n})+u_{1}(Y_{n})\right)+\left(g(Y_{n})+u_{2}(Y_{n})\right)\left(W(t)-W(t_{n})\right)\right|^{p}\nonumber\\
&\le& 2^{p-1}\left(\mathbb{E}\left|(t-t_{n})\left(f(Y_{n})+u_{1}(Y_{n})\right)\right|^{p}+\mathbb{E}\left|\left(g(Y_{n})+u_{2}(Y_{n})\right)\left(W(t)-W(t_{n})\right)\right|^{p}\right)\nonumber\\
&\le& 2^{2p}K^{p}\mathbb{E}|Y_{n}|^{p}h^{\frac{p}{2}}
\end{eqnarray*}
Similarly as the proof of \Cref{lem21}, we have
\begin{eqnarray}\label{eqA27}
\sup_{0\le s\le t}\mathbb{E}|y_{\Delta}(s)|^{p}\le |x_{0}|^{p}e^{2pK\left(1+(p-1)K\right)t}
\end{eqnarray}
Applying \eqref{eqA27}, we obtain
\begin{eqnarray}\label{eqA112}
\mathbb{E}|y_{\Delta}(t)-\bar{y}_{\Delta}(t)|^{p}&\le& 2^{2p}K^{p}e^{2pK\left(1+(p-1)K\right)t}h^{\frac{p}{2}}|x_{0}|^{p}
\end{eqnarray}
Substituting \eqref{eqA112} into \eqref{eqA107}, we obtain
\begin{eqnarray*}
\mathbb{E}|y(t)-y_{\Delta}(t)|^{p}&\le&\left(2K(2p-1)+8(p-1)^{2}K^{2}\right)\int_{0}^{t}\mathbb{E}|y(s)-y_{\Delta}(s)|^{p}ds\nonumber\\
&&+\left(2K+8(p-1)K^{2}\right)\int_{0}^{t}2^{2p}K^{p}e^{2pK\left(1+(p-1)K\right)s}h^{\frac{p}{2}}|x_{0}|^{p}ds
\end{eqnarray*}
By Gronwall inequality, we have
\begin{equation*}
\sup_{0\le t\le T}\mathbb{E}|y(t)-y_{\Delta}(t)|^{p}\le H_{9}(T, p, K)h^{\frac{p}{2}}|x_{0}|^{p},
\end{equation*}
where 
\begin{equation}\label{H9}
H_{9}(T, p, K)=[1+4(p-1)K]2^{2p+1}K^{p+1}e^{2KT(3p-1+(p-1)(5p-4)K)}T.
\end{equation}
By noting $y_{\Delta}(t_{n})=Y_{n}$, we get for $t=t_{n}$
\begin{equation*}
\sup_{0\le t_{n}\le T}\mathbb{E}|y(t_{n})-Y_{n}|^{p}\le H_{9}(T, p, K)h^{\frac{p}{2}}|x_{0}|^{p}
\end{equation*}
The proof is completed.
\end{proof}

\begin{theorem}\label{theorem 6}
Let \Cref{assumption 1} hold. Assume that the SDE \eqref{SDEs} is $p$th moment exponentially stable and satisfies \eqref{exp-sta-y}. Let $T=1+4\ln(2^{p-1}M_{2})/\gamma_{2}$. 
If $h$ satisfies 
\begin{equation}\label{A13}
2^{p-1}H_{9}(2T, p, K)h^{\frac{p}{2}}+e^{-\frac{3}{4}\gamma_{2}T}\le e^{-\frac{1}{2}\gamma_{2}T}.
\end{equation}
 Then the EMSDE \eqref{SDEs-Euler-1} is $p$th moment exponentially stable.
\end{theorem}
\begin{proof}
Since $T=1+4\ln(2^{p-1}M_{2})/\gamma_{2}$, we have
\begin{equation*}
2^{p-1}M_{2}e^{-\gamma_{2}T}< e^{-\frac{3}{4}\gamma_{2}T}.
\end{equation*}
Now, for any given $i\in\mathbb{N}$, let $\{\hat{y}(t)\}_{t\ge iT}$ be the solution to the SDE \eqref{SDEs} for $t\in[iT, \infty)$, with the initial condition $y_{\Delta}(iT)$. Then using basic inequality, Lemma \ref{lemA1}, \eqref{exp-sta-y} and \eqref{A13}, we have
\begin{eqnarray}\label{eqA17}
&&\sup_{(i+1)T\le t\le (i+2)T}\mathbb{E}|y_{\Delta}(t)|^{p}\nonumber\\
&\le& 2^{p-1}\sup_{iT\le t\le (i+2)T}\mathbb{E}|y_{\Delta}(t)-\hat{y}(t)|^{p}+2^{p-1}\sup_{(i+1)T\le t\le (i+2)T}\mathbb{E}|\hat{y}(t)|^{p}\nonumber\\
&\le& \left(2^{p-1}H_{9}(2T, p, K)h^{\frac{p}{2}}+2^{p-1}M_{2}e^{-\gamma_{2}T}\right)\mathbb{E}|y_{\Delta}(iT)|^{p}\nonumber\\
&\le& \left(2^{p-1}H_{9}(2T, p, K)h^{\frac{p}{2}}+e^{-\frac{3}{4}\gamma_{2}T}\right)\mathbb{E}|y_{\Delta}(iT)|^{p}\nonumber\\
&\le& e^{-\frac{1}{2}\gamma_{2}T}\sup_{iT\le t\le (i+1)T}\mathbb{E}|y_{\Delta}(t)|^{p}.
\end{eqnarray}
According to \eqref{eqA27},
\begin{eqnarray}\label{eqA15}
\sup_{0\le t\le T}\mathbb{E}|y_{\Delta}(t)|^{p}\le |x_{0}|^{p}e^{2pK\left(1+(p-1)K\right)T}\le L_{2}e^{-\frac{1}{2}\gamma_{2}t}|x_{0}|^{p},
\end{eqnarray}
where $L_{2}=e^{\frac{1}{2}\gamma_{2}T+2pK\left(1+(p-1)K\right)T}$. Combining \eqref{eqA15} and \eqref{eqA17}, we obtain that
\begin{eqnarray}\label{eqA28}
\sup_{(i+1)T\le t \le(i+2)T}\mathbb{E}|y_{\Delta}(t)|^{p}&\le& e^{-\frac{1}{2}(i+1)\gamma_{2}T}\sup_{0\le t\le T}\mathbb{E}|y_{\Delta}(t)|^{p}\nonumber\\
&\le& e^{-\frac{1}{2}(i+1)\gamma_{2}T}|x_{0}|^{p}e^{2pK\left(1+(p-1)K\right)T}\nonumber\\
&\le& L_{2}e^{-\frac{1}{2}\gamma_{2}t}|x_{0}|^{p}.
\end{eqnarray}
 Due to \eqref{eqA28} and \eqref{eqA15}, the proof is completed by using $t=t_{n}$.
\end{proof}

\section{Conclusion}
In this paper, we have shown from \Cref{theorem 3}, \Cref{theorem 4}, \Cref{Q6} and \Cref{theorem 6} that, under the standing \Cref{assumption 1},
\begin{eqnarray*}
SDE \eqref{SDEs}\stackrel{Q2}{\to}\  EMSDE \eqref{SDEs-Euler-1}\stackrel{Q3}{\to}\  EMSDEPCA \eqref{SDEPCAs-Euler-1}\stackrel{Q4}{\to}SDEPCA \eqref{SDEPCAs}\stackrel{Q1}{\to}SDE \eqref{SDEs}.
\end{eqnarray*}
 Hence we have the following theorem.
\begin{theorem}\label{thm61}
Under \Cref{assumption 1}, if one of SDEPCA \eqref{SDEPCAs}, SDE \eqref{SDEs}, EMSDEPCA \eqref{SDEPCAs-Euler-1} and EMSDE \eqref{SDEs-Euler-1} is $p$th moment exponentially stable, then the other three are also $p$th moment exponentially stable for sufficiently small step size $h$ and $\tau$.
\end{theorem}

By examming the proof of the \Cref{theorem 3}, \Cref{theorem 4}, \Cref{Q6} and \Cref{theorem 6}, we see that the $p$th moment exponential stability of  SDEPCA \eqref{SDEPCAs}, SDE \eqref{SDEs}, EMSDEPCA \eqref{SDEPCAs-Euler-1} and EMSDE \eqref{SDEs-Euler-1} are equivalent as long as their solutions are $p$th moment bounded and arbitrarily close for sufficiently small $\tau$ and $h$. 
Let 
\begin{eqnarray*}
F(y(t))=f(y(t))+u_{1}(y(t))\quad and\quad G(y(t))=g(y(t))+u_{2}(y(t))
\end{eqnarray*}
For $V\in C^{2,1}(\mathbb{R}^{d}\times\mathbb{R}_{+}; \mathbb{R}_{+})$, we define an operator $\mathcal{L}V$ by
\begin{equation*}
\mathcal{L}V(y, t)=V_{t}(y, t)+V_{y}(y, t)F(y(t))+\frac{1}{2}trace\left[G^{T}(y)V_{yy}(y, t)G(y)\right].
\end{equation*}
The sufficient criterion for $p$th moment exponential stability via a Lyapunov function is given by Theorem 4.4 in \cite[][P130]{Mao2008}. Now we quote it here.
\begin{theorem}\label{thm62}
Assume that there is a function $V(y, t)\in C^{2, 1}(\mathbb{R}^{d}\times\mathbb{R}_{+}; \mathbb{R}_{+})$, and positive constants $c_{1}, c_{2}, c_{3}$ such that 
\begin{equation*}
c_{1}|y|^{p}\le V(y, t)\le c_{2}|y|^{p}\quad and \quad \mathcal{L}V(y, t)\le -c_{3}V(y, t)
\end{equation*}
for all $(y, t)\in\mathbb{R}^{d}\times\mathbb{R}_{+}$. Then for the SDE \eqref{SDEs}, we have 
\begin{equation*}
\mathbb{E}|y(t)|^{p}\le \frac{c_{2}}{c_{1}}|x_{0}|^{p}e^{-c_{3}t},
\end{equation*}
for all $x_{0}\in\mathbb{R}^{d}$. In other words, the SDE \eqref{SDEs} is $p$th moment exponentially stable.
\end{theorem}
For convenience, we impose the following hypothesis.
\begin{assumption}\label{asp63}
There exists a pair of positive constants $p$ and $\lambda$ such that
\begin{equation*}
|y|^{2}\left(2y^{T}F(y)+|G(y)|^{2}\right)-(2-p)|y^{T}G(y)|^{2}\le -\lambda|y|^{4},\quad \forall \ y\in\mathbb{R}^{d}.
\end{equation*}
\end{assumption}
Applying the \Cref{thm62} with $V(y, t)=|y|^{p}$, we easily obtain the following theorem \cite[see][]{LiXiaoyue2019}.
\begin{theorem}
Under \Cref{asp63}, the SDE \eqref{SDEs} is $p$th moment exponentially stable, i.e.
\begin{equation*}
\mathbb{E}|y(t)|^{p}\le |x_{0}|^{p}e^{-\frac{\lambda}{2}pt},\quad \ \forall \ t>0,
\end{equation*}
where $p$ and $\lambda$ are given in \Cref{asp63}.
\end{theorem}
In combination with \Cref{thm61}, the following theorem provides an interesting result.
\begin{theorem}
Assume that \Cref{assumption 1} and \Cref{asp63} hold, then SDE \eqref{SDEs} is $p$th moment exponentially stable and SDEPCA \eqref{SDEPCAs}, EMSDEPCA \eqref{SDEPCAs-Euler-1}, EMSDE \eqref{SDEs-Euler-1} are also $p$th moment exponentially stable as long as step size $h$ and $\tau$ are sufficiently small.
\end{theorem}

%


\bibliographystyle{siamplain}
\bibliography{References_o}
\end{document}